\documentclass[11pt]{amsart}
\usepackage{amssymb, amscd}
\usepackage[all]{xy}
\usepackage{comment}

\theoremstyle{plain}
\newtheorem{thm}{Theorem}[section]
\newtheorem{prop}[thm]{Proposition}
\newtheorem{defn}[thm]{Definition}

\newtheorem{que}[thm]{Question}

\theoremstyle{remark}

\numberwithin{equation}{section}

\newcommand{\vol}{\mathrm{vol}}
\newcommand{\nd}{\mathrm{nd}}

\newcommand{\de}{\partial}

\newcommand{\R}{\mathbb{R}}

\renewcommand{\leq}{\leqslant}
\renewcommand{\geq}{\geqslant}
\renewcommand{\le}{\leqslant}

\begin{document}

\numberwithin{equation}{section}

\title{Volume of perturbations of pseudoeffective classes}
\begin{abstract} In this short note, we consider the question of determining the asymptotics of the volume function near the boundary of the pseudoeffective cone on compact K\"ahler manifolds. We solve the question in a number of cases -- in particular, we show that the volume function behaves polynomially under small perturbations near pseudoeffective classes with numerical dimension zero.
\end{abstract}
\author[N. McCleerey]{Nicholas McCleerey}
\email{njm2@math.northwestern.edu}
\address{Department of Mathematics, Northwestern University, 2033 Sheridan Road, Evanston, IL 60208}

\thanks{Partially supported by NSF RTG grant DMS-1502632.}

\maketitle
\section{Introduction}
Let $(X^n,\omega)$ be a compact K\"ahler manifold, and $\alpha$ a closed real $(1,1)$ form on $X$, whose cohomology class $[\alpha]$ is pseudoeffective, i.e. it contains some closed positive $(1,1)$ current. The set of all such classes form a closed cone $\mathcal E(X)\subset H^{1,1}(X,\R)$ called the pseudoeffective cone, and it is known that the volume function,
\[
\vol(\alpha) := \sup_{0\leq T\in[\alpha]} \int_X T_{\mathrm{ac}}^n
\]
as defined for cohomology classes in \cite{Bo}, has the property that $\mathrm{vol}(\alpha) > 0$ if and only if $\alpha$ is in the interior of $\mathcal E(X)$, in which case we say that $[\alpha]$ is a big class \cite{Bo}. Furthermore, the volume function is continuous on all of $\mathcal E(X)$.

When $X$ is projective and $[\alpha]=c_1(L)$ for some holomorphic line bundle $L$, then Boucksom \cite{Bo} showed that
$$\vol(c_1(L))=\lim_{m\to\infty}\frac{h^0(X,L^m)}{m^n/n!}=:\vol(L),$$
namely the volume agrees with the algebraic definition (see Lazarsfeld's monograph \cite{Laz} for more on the volume of line bundles).

In this paper, we would like to investigate the asyptotics of the volume function near the boundary of the pseudoeffective cone. More precisely, if $[\alpha]\in\partial\mathcal E(X)$ (which we shall assume from now on), we would like to study the behavior of
\[
\mathrm{vol}(\alpha + t\omega)
\]
as $t > 0$ tends to zero. As mentioned above, the fact that $[\alpha]\in\partial\mathcal E(X)$ implies that $\mathrm{vol}(\alpha + t\omega)\to 0$ as $t\searrow 0$, and we would like to understand the rate at which this approaches zero. For example, if $[\alpha]$ is nef (i.e. a limit of K\"ahler classes), then Boucksom \cite{Bo} showed that for all $t\geq 0$ we have
\begin{equation}\label{nientesaleinzucca}
\vol(\alpha+t\omega)=\int_X(\alpha+t\omega)^n=t^{n-k}\binom{n}{k}\int_X\alpha^k\wedge\omega^{n-k}+O(t^{n-k+1}),
\end{equation}
where $0\leq k<n$ is the largest nonnegative integer such that $\int_X\alpha^k\wedge\omega^{n-k}\neq 0$ (or equivalently, such that $[\alpha^k]\neq 0$ in $H^{k,k}(X,\mathbb{R})$). This integer, denoted by $\nd(\alpha)$, is called the {\em numerical dimension} of the nef class $[\alpha]$, and \eqref{nientesaleinzucca} shows that $\vol(\alpha+t\omega)\sim t^{n-\nd(\alpha)}$ when $[\alpha]$ is nef.

When $[\alpha]$ is merely pseudoeffective, there are a number of natural notions of numerical dimension of $[\alpha]$, starting from the algebraic work of Nakayama \cite{Nak} and of Boucksom-Demailly-P\u{a}un-Peternell \cite{BDPP} on K\"ahler manifolds, and several inequalities relating them were proved by Lehmann \cite{Le} and Eckl \cite{Ec}. We will consider one such notion, introduced in \cite{BDPP}, which is the direct analog of what happens in the nef case, namely
$$\nd(\alpha):=\max\{k\in\mathbb{N}\ |\ \langle\alpha^k\rangle\neq 0 \textrm{ in }H^{k,k}(X,\mathbb{R})\},$$
where $\langle \alpha^k \rangle$ is the positive intersection product of Boucksom \cite{BDPP} (see also \cite{BEGZ} in the transcendental case). When $[\alpha]$ is nef we have $\langle\alpha^k\rangle=[\alpha^k]$, so this is consistent with the definition in the nef case. Also, we have that $\int_X\langle\alpha^n\rangle=\vol(\alpha)$, so if $[\alpha]\in\de\mathcal E(X)$ then we have $0\leq \nd(\alpha)<n$.

Insipired by what happens in the nef case, we first study the following question:
\begin{que}\label{con}
Let $(X^n,\omega)$ be a compact K\"ahler manifold and $[\alpha]$ a pseudoeffective $(1,1)$ class with $\vol(\alpha)=0$.
Do we have that
$$\vol(\alpha+t\omega)=O(t^{n-\nd(\alpha)}),$$
as $t>0$ approaches zero?
\end{que}

As mentioned above, the answer is affirmative if $[\alpha]$ is nef. Not surprisingly, the answer is also affirmative when $n=2$, see Proposition \ref{surfaces} below. Our main result is the following:

\begin{thm}\label{main}
Question \ref{con} has an affirmative answer if either:
\begin{itemize}
\item[(a)] $\nd(\alpha)=n-1$, and the volume function $\vol(\cdot)$ is differentiable in the big cone, or
\item[(b)] $\nd(\alpha)=0$.
\end{itemize}
\end{thm}
In the first item (a), let us remark that differentiability of the volume function is known to hold on projective manifolds by \cite{WN,BDPP,BFJ}, and is conjectured to be true on arbitrary compact K\"ahler manifolds \cite{BDPP}.

However, starting from $n=3$, counterexamples to Question \ref{con} were very recently constructed by Lesieutre in \cite{Les}. Specifically, he constructs a Calabi-Yau $3$-fold $X$ with a class $[\alpha]$ as above such that $\nd(\alpha)=1$ and
$$\vol(\alpha+t\omega)=O(t^{\frac{3}{2}}).$$
This naturally raises the question whether there is a potentially new notion of numerical dimension coming from the asymptotics of the volume function, and we discuss this briefly in Section \ref{end}.

This paper is organized as follows. In the next section, we start with some simple initial observations, then prove item (a) in Theorem \ref{main}, and also the fact that Question \ref{con} has an affirmative answer on surfaces. In section \ref{zero} we then deal with item (b), when the numerical dimension is zero, and show that the volume function is actually polynomial along small perturbations near such classes. Finally, in Section \ref{end} we briefly discuss a possible direction of further inquiry concerning the recent paper \cite{Les}.\newline

\noindent
{\bf Acknowledgments. } I would like to thank John Lesieutre for sharing his recent preprint \cite{Les} with me while preparing this note and for his many helpful suggestions regarding it, S\'ebastien Boucksom for the proof of Proposition \ref{surfaces}, and Jean-Pierre Demailly, Rob Lazarsfeld, David Witt Nystr\"om and Jian Xiao for related discussions. I would also like to thank my advisor Valentino Tosatti for his continued patience and guidance.

\section{Preliminary Observations}
Throughout this paper, $(X^n,\omega)$ will be a compact K\"ahler manifold, and $\alpha$ a closed real $(1,1)$ form on $X$ whose cohomology class $[\alpha]\in H^{1,1}(X,\mathbb{R})$ is pseudoeffective, but not big, so that $\vol(\alpha) = 0$. Monotonicity and homogeneity of the volume then imply immediately that one always has the following lower bound:
\[
\vol(\alpha + t\omega) \geq \vol(\omega)t^n = \left(\int_X\omega^n\right)\ t^n.
\]
Moreover, if we assume that the volume function $\vol(\cdot)$ is differentiable in the big cone (which is satisfied on all projective manifolds by \cite{WN,BDPP,BFJ} and conjectured to be always true \cite{BDPP}), one always has the following upper-bound:
\begin{prop}\label{Diff}
Let $(X^n,\omega)$ be a compact K\"ahler manifold and $[\alpha]$ a pseudoeffective $(1,1)$ class with $\vol(\alpha)=0$.
Assume that $\vol(\cdot)$ is differentiable on the big cone. Then there exists $C>0$ such that
\[
 \vol(\alpha + t\omega) \leq C t,
\]
for all $t\geq 0$ sufficiently small.
\end{prop}
\begin{proof}
Thanks to Boucksom-Favre-Jonsson \cite{BFJ}, the assumption of differentiability of $\vol(\cdot)$ implies that
$$\frac{d}{dt}\bigg|_{t=t_0}\vol(\alpha+t\omega)=n\langle (\alpha + t_0\omega)^{n-1}\rangle \cdot\omega.$$
Using this together with the fundamental theorem of calculus and the monotonicity of the positive intersection product, we have:
\[
 \vol(\alpha + t\omega) - \vol(\alpha) = n\int_{0}^t \langle(\alpha + s\omega)^{n-1}\rangle \cdot \omega\ ds\leq n\left(\langle (\alpha + \omega)^{n-1}\rangle \cdot\omega\right) t.
\]
Hence, since $\vol(\alpha) = 0$, we get:
\[
 \vol(\alpha + t\omega) \leq C t,
\]
for all $0 < t \leq 1$.
\end{proof}
Recall that the numerical dimension of a pseudoeffective $(1,1)$ class is defined by \cite{BDPP} to be:
\begin{defn}
The numerical dimension of a pseudoeffective $(1,1)$ class $[\alpha]$ is defined to be
$$\nd(\alpha):=\max\{p\in\mathbb{N}\ |\ \langle\alpha^p\rangle\neq 0 \textrm{ in }H^{p,p}(X,\mathbb{R})\},$$
where $\langle \alpha^p \rangle$ is the positive intersection product of Boucksom \cite{BDPP,BEGZ}. In particular, if $[\alpha]$ is not big then we have $0\leq \nd(\alpha)\leq n-1$.
\end{defn}
It has then been established in \cite{Le} that the lower bound for the volume is actually directly related to $\nd(\alpha)$. We reproduce the following short proposition verbatim from \cite[Theorem 6.2]{Le} (noting that it applies to general compact K\"ahler manifolds and $(1,1)$ classes), for the reader's convenience:

\begin{prop}\label{Leh}
We have that
$$\nd(\alpha)=\max\{p\in\mathbb{N}\ |\ \exists c>0 \textrm{ such that }\vol(\alpha+t\omega)\geq c t^{n-p}, \textrm{ for all }t>0\}.$$
\end{prop}
\begin{proof}
For any $p\geq 0$ we have
$$t^{n-p}\langle(\alpha+t\omega)^p\rangle \cdot \omega^{n-p}=\langle (\alpha+t\omega)^p\cdot (t\omega)^{n-p}\rangle\leq \langle (\alpha+t\omega)^n\rangle=\vol(\alpha+t\omega),$$
and so we conclude that $\vol(\alpha+t\omega)\geq ct^{n-\nd(\alpha)}$ for some $c>0$ and all $t>0$. This shows that
$$\nd(\alpha)\geq\max\{p\in\mathbb{N}\ |\ \exists c>0 \textrm{ such that }\vol(\alpha+t\omega)\geq c t^{n-p}, \textrm{ for all }t>0\}.$$
Conversely, if $p$ is the maximum on the RHS, then for every $c>0$ there is some $t>0$ such that $\vol(\alpha+t\omega)<ct^{n-p-1},$ and so for this value of $t$ we have
$$ct^{n-p-1}>\vol(\alpha+t\omega)\geq t^{n-p-1}\langle(\alpha+t\omega)^{p+1}\rangle \cdot \omega^{n-p-1},$$
i.e. $$\langle(\alpha+t\omega)^{p+1}\rangle \cdot \omega^{n-p-1}<c.$$
But the LHS of this is increasing in $t$, and so this inequality holds for all $t>0$ sufficiently small, and letting $t$ tend to zero gives
$$\langle \alpha^{p+1}\rangle \cdot \omega^{n-p-1}<c.$$
Since $c>0$ is arbitrary, we conclude that
$$\langle \alpha^{p+1}\rangle \cdot \omega^{n-p-1}=0,$$
and since $\langle \alpha^{p+1}\rangle$ is represented by a strongly positive $(p+1,p+1)$ current, we conclude that
$\langle\alpha^{p+1}\rangle=0$ in cohomology, as required.
\end{proof}

By combining Propositions \ref{Diff} and \ref{Leh}, we immediately deduce item (a) of our main theorem \ref{main}.

Next, we show that Question \ref{con} has an affirmative answer of all (not necessarily projective) surfaces. The following proof was communicated to us by S. Boucksom:

\begin{prop}\label{surfaces}
Question \ref{con} has an affirmative answer when $n=2$.
\end{prop}
\begin{proof}
To see this, let $\alpha=P+N$ and $\alpha+t \omega=P_t+N_t$ be the Zariski decompositions of $\alpha$ and $\alpha + t\omega$, which always exist on surfaces \cite{Zar}. Let $E_i$ be the irreducible components of the non-K\"ahler locus $E_{nK}(\alpha)$ (see \cite{Bo2}), so that:
\[
N = \sum_i \nu(\alpha, E_i) E_i,
\]
where $\nu(\alpha,E_i)$ is the minimal multiplicity of $\alpha$ along $E_i$ \cite[Def. 3.1]{Bo2}. Now it is clear that $N_t\le N$, so we also have:
\[
N_t= \sum_i \nu(\alpha + t\omega, E_i) = \sum_i (\nu(\alpha, E_i) -a_i(t)) E_i
\]
for nonnegative constants $a_i(t)\to 0$, by lower semicontinuity of the minimal multiplicity \cite[Prop. 3.5]{Bo2}. Indeed, for all $i$ we have
$$\nu(\alpha,E_i)\leq\liminf_t \nu(\alpha+t\omega,E_i)\leq \limsup_t\nu(\alpha+t\omega,E_i)\leq \nu(\alpha,E_i),$$
where the last inequality follows immediately from the definition of $\nu$.

It follows now that, for $t$ small enough, $P_t=P+t \omega+\sum a_i(t) E_i$ is orthogonal to each $E_i$, and since $P$ itself is orthogonal to each $E_i$, we may dot both sides against $E_j$ to get that:
\[
\sum_i a_i(t) E_i\cdot E_j=-t \omega\cdot E_j.
\]
By non-degeneracy of the Gram matrix $(E_i\cdot E_j)$, it thus follows that $a_i(t)=O(t)$, and hence
\[
\vol(\alpha+t \omega)=(P_t^2)= \left(P+t \omega+\sum a_i(t) E_i\right)^2 =(P+t \omega)^2+O(t^2),
\]
using again that $P$ is orthogonal to each $E_i$. Finally, note that $P^2 = \vol(\alpha) = 0$ and that (by definition) $P=\langle[\alpha]\rangle$, so that:
\[
(P + t\omega)^2 = 2(P\cdot \omega) t + (\omega^2) t^2 = c t^{2-\nd(\alpha)}+O(t^2)
\]
as desired.
\end{proof}

\section{Numerical Dimension Zero}\label{zero}

In this section, we deal with the case when the class $[\alpha]$ has numerical dimension zero. Recall that having $\nd(\alpha)=0$ is equivalent to having that $[\alpha]= N(\alpha)$,  where here $N(\alpha)$ is the negative part in the divisorial Zariski decomposition of \cite{Bo2} (see also \cite{Nak} for the algebraic case). In particular, $N(\alpha)$ is the cohomology class of some effective $\R$ divisor $D$.
\begin{prop}
Let $X^n$ be a compact K\"ahler manifold, let $[\alpha]=[D]$ be a $(1,1)$-class with $\nd(\alpha)=0$, and let $[\beta]$ be any other $(1,1)$ class. Then there exists a constant $C\geq 0$ such that for all $t > 0$ sufficiently small we have:
\[
\mathrm{vol}([D] + t\beta) = Ct^n.
\]
In particular, Question \ref{con} has an affirmative answer when $\nd(\alpha)=0$.
\end{prop}
\begin{proof}
If $D+t\beta$ is not big for all small $t > 0$, then we may simply set $C = 0$. Otherwise, we may assume that $D + t\beta$ is big for all sufficiently small $t > 0$. Note now that the proposition is equivalent to asking that $\vol(\frac{1}{t}[D] + \beta) = C$ for all $t$ sufficiently small. 

To this end, suppose we show that
\begin{equation}\label{goal}
\mathrm{Supp}(D)\subseteq E_{nK}([D] + t\beta) = E_{nK}\left(\frac{1}{t}[D] + \beta\right),
\end{equation}
for all $0<t\leq t_0$ sufficiently small. Then for all $0<t<t_0$ we can apply \cite[Theorem 3.7]{DFT} to the class $\frac{1}{t_0}[D] + \beta$ and conclude that
$$\vol\left(\frac{1}{t}[D] + \beta\right)=\vol\left(\frac{1}{t_0}[D] + \beta+\left(\frac{1}{t}-\frac{1}{t_0}\right)[D]\right)=\vol\left(\frac{1}{t_0}[D] + \beta\right),$$
which is indeed constant as $t$ varies.

To prove \eqref{goal}, we shall first deal with the case when $\beta = \omega$ is a K\"ahler class. It follows from \cite[Def. 3.7]{Bo2} that, if $D$ is a divisor with $\nd(D)=0$, then we have,
\[
D = N(D) = \sum_i \nu(D, E_i) E_i,
\]
where the $E_i$ are the irreducible components of $D$. Thus, by \cite[Def. 3.3]{Bo2}, we must have:
\[
\mathrm{Supp}(D) = E_{nn}(D) = \mathbb B_-(D),
\]
where $E_{nn}(D)$ is the non-nef locus \cite{Bo2}, which equals the diminished base locus $\mathbb B_-(D)$ of \cite{ELMNP}. Now, it is well-known (cf. \cite{ELMNP}) that we can also characterize the non-nef locus as: $$E_{nn}(D) = \bigcup_{t > 0} E_{nK}([D] + t\omega).$$ It is then easy to see that the subvarieties $E_{nK}([D] + t\omega)$ are decreasing in $t$, and since their union is the proper subvariety $\mathrm{Supp}(D)$, they must stabilize as $t$ goes to zero at some $t_0$; i.e., for all $t \leq t_0$, we actually have $$E_{nK}([D] + t\omega) = E_{nn}(D) = \mathrm{Supp}(D),$$ which was to be shown.

When now $\beta$ is arbitrary, we simply choose a large enough K\"ahler class $\omega$ so that $\beta + \omega$ is also K\"ahler. Then for any $t > 0$, $$E_{nK}\left(\frac{1}{t}[D] + \beta + \omega\right)\subseteq E_{nK}\left(\frac{1}{t}[D] + \beta\right),$$ and so for $t \leq t_0$ as above, we have $\mathrm{Supp}(D)\subseteq E_{nK}([D] + t\beta)$, as was to be shown.
\end{proof}

\section{Concluding questions}\label{end}
As remarked earlier, despite the positive results in Theorem \ref{main}, the answer to Question \ref{con} is negative in general, thanks to a very recent counterexample of Lesieutre \cite{Les}. Note that his example has $n=3,\nd(\alpha)=1$, which is the first case which is not covered by Theorem \ref{main}.

One can however ask then the following question:

\begin{que}[Lesieutre \cite{Les}]\label{aaa}
Let $(X^n,\omega)$ be a compact K\"ahler manifold and $[\alpha]$ a pseudoeffective $(1,1)$ class. Does there exist a positive {\bf real} number $\nd_{\mathrm{vol}}(\alpha)$ such that
$$ct^{n - \nd_{\mathrm{vol}}(\alpha)}\leq \vol(\alpha+t\omega)\leq Ct^{n - \nd_{\mathrm{vol}}(\alpha)},$$
 for some $c, C>0$ and for all $t\geq0$ sufficiently small?
\end{que}

It is immediate that, if $\nd_{\mathrm{vol}}(\alpha)$ exists, we would have $\nd(\alpha) \leq \nd_{\mathrm{vol}}(\alpha) \leq \nd(\alpha) + 1$, and so it would be intermediate amongst the various different notions of numerical dimension. All known examples, including those in \cite{Les}, admit such a number.

A possible intermediate step in answering Question \ref{aaa}, also suggested by Lesieutre, would be the following:
\begin{que}[Lesieutre \cite{Les}]\label{aaaa}
Let $(X^n,\omega)$ be a compact K\"ahler manifold and $[\alpha]$ a pseudoeffective $(1,1)$ class.
Then the limit
\[
\lim_{t\searrow 0} \frac{\log(\vol(\alpha + t\omega))}{\log(t)},
\]
exists.
\end{que}
This is generally weaker than Question \ref{aaa}, but the limit would compute $\nd_{\mathrm{vol}}(\alpha)$, if it did exist. Note that one cannot simply use log-concavity of the volume, as it is not true in general that any two concave functions can only intersect a finite number of times on a compact interval.

\end{document}